\documentclass[11pt,a4paper,twoside,reqno]{amsart}
\usepackage[
a4paper,
left=4cm, right=4cm,
top=4cm,  bottom=4cm,
includeheadfoot,   
headheight=14pt,   
footskip=18pt      
]{geometry}
\usepackage[all]{xy}
\usepackage{amsmath,amssymb,amsfonts,amsthm}
\usepackage[toc,page]{appendix}
\usepackage{verbatim} 
\usepackage{graphicx} 
\usepackage{amsfonts}
\usepackage{amsthm}
\usepackage{amsmath}
\usepackage{stmaryrd}
\usepackage{amscd}
\usepackage{mathtools}
\usepackage[latin2]{inputenc}
\usepackage{t1enc}
\usepackage[mathscr]{eucal}
\usepackage{indentfirst}
\usepackage{graphicx}
\usepackage{graphics}
\usepackage{pict2e}
\usepackage{epic}
\usepackage{faktor}
\usepackage{xfrac}
\numberwithin{equation}{section}
\usepackage{epstopdf}
\usepackage{xurl}
\usepackage{hyperref}
\usepackage{lipsum}
\usepackage{tikz-cd}
\usepackage{pict2e}
\usepackage{epic}
\usepackage{float}
\usepackage{caption}
\usepackage{amssymb}
\numberwithin{equation}{section}
\usepackage{epstopdf}
\usepackage{comment}
\usepackage{tikz-cd}
\usepackage{microtype}
\usepackage{needspace}

\usepackage{xcolor}
\usepackage{enumitem}
\usepackage[nameinlink,noabbrev]{cleveref}

\hypersetup{
colorlinks=true,
linkcolor=blue,
citecolor=blue,
urlcolor=blue,
pdftitle={Minimal Foliations, Codimension-One Stable Norms, and a Question of Bangert},
pdfauthor={Hoan Nguyen}
}
\theoremstyle{plain}
\newtheorem{Th}{Theorem}[section]

\theoremstyle{definition}

\newtheorem{?}[Th]{Problem}

\theoremstyle{plain}
\newtheorem{theorem}[Th]{Theorem}
\newtheorem{lemma}[Th]{Lemma}
\newtheorem{corollary}[Th]{Corollary}

\newtheorem{question}[Th]{Question}
\theoremstyle{definition}

\theoremstyle{remark}
\newtheorem{remark}[Th]{Remark}

\newcommand{\T}{\mathbb T}
\newcommand{\R}{\mathbb R}
\newcommand{\Z}{\mathbb Z}
\newcommand{\Sc}{\operatorname{Sc}}
\newcommand{\PD}{\operatorname{PD}}
\newcommand{\Vol}{\operatorname{Vol}}
\newcommand{\Mass}{\mathbf M}

\newcommand{\st}{\mathrm{st}}
\newcommand{\Id}{\operatorname{Id}}

\title[Codimension-one stable norm]{Minimal Foliations, Codimension-One Stable Norms, and a Question of Bangert}
\author{Hoan Nguyen}
\address{Department of Mathematics, The University of Chicago, Chicago, IL 60615}
\email{nth@uchicago.edu}
\date{}

\numberwithin{equation}{section}

\begin{document}

\begin{abstract}
We compute the codimension-one stable norm for a natural class of
cohomogeneity-one metrics on tori.  In every dimension $n\ge3$, the
formula yields smooth nonflat metrics for which each primitive
codimension-one homology class is represented by a foliation of
calibrated tori, giving a negative answer to a question of Bangert.
On $\mathbb T^3$, we construct an infinite-dimensional family of
nonflat metrics whose codimension-one stable norm agrees exactly with
that of the unit cubic flat torus and whose total volume is fixed.
An explicit two-parameter subfamily contains pairwise non-isometric metrics.  These examples
also show that the Euclidean-stable-norm-and-volume data are not locally injective
near the cubic flat metric.  Conversely, among smooth metrics on $\T^3$
admitting a free isometric circle action and having the cubic Euclidean
codimension-one stable norm, we prove that volume is at most one, with
equality only for the cubic flat metric up to an isometry isotopic to the
identity.
\end{abstract}

\maketitle

\section{Introduction}

In his contribution to the 1994 International Congress of Mathematicians, Bangert asked whether an abundance of codimension-one minimizing hypersurfaces characterizes flat Riemannian tori \cite{Bangert1995}. More precisely, he posed the following question.
\begin{?}
    Suppose that $(\mathbb T^n, g)$ is a Riemannian torus such that for every primitive class $r\in H_{n-1}(\mathbb T^n, \mathbb Z)$, there exists a foliation of $\mathbb T^n$ by $g$-minimal $(n-1)$-tori in the class $r$. Must $g$ be a flat metric?
\end{?}
The motivation comes from the geodesic case on $\mathbb T^2$, where the answer is positive. Suppose that for every primitive class in $H_1(\mathbb T^2, \mathbb Z)$, the torus is foliated by closed geodesics representing that class. Each such foliation lifts to a foliation of $\mathbb R^2$ by globally minimizing geodesics. Innami \cite{Innami1986} proved that this condition forces the metric to be flat. This is closely related to Hopf's theorem \cite{Hopf1948} that a Riemannian two-torus without conjugate points is flat. The corresponding rigidity theorem in every dimension
was proved by Burago and Ivanov \cite{BuragoIvanov1994}.

For hypersurfaces in higher-dimensional $\mathbb T^n$, minimal hypersurfaces in such a foliation are still homologically area-minimizing. This is because any such foliation is calibrated: if $\nu$ is its global unit normal, then
\[
  d(\iota_\nu dV_g)=(\operatorname{div}_g\nu)dV_g=0,
\]
hence $\iota_\nu dV_g$ is a calibration form on every leaf. Bangert's
hypothesis can therefore be stated equivalently using foliations by
homologically area-minimizing tori.

Moser initiated a higher-dimensional Aubry-Mather theory for periodic
variational problems and proved a stability theorem for certain minimal
foliations \cite{Moser1986,Moser1988}.  Bangert developed the corresponding
structure theory of non-selfintersecting minimizers, introducing secondary
invariants and describing the resulting foliations and laminations with
gaps on the torus \cite{Bangert1989}.

The natural invariant associated with these minimizing hypersurfaces is
the codimension-one stable norm.  For a closed oriented Riemannian
$n$-manifold $(M,g)$ and a class
$ r\in H_{n-1}(M,\R)$, 
its stable norm is defined by the minimum mass of a real $(n-1)$-cycle representing $r$.
This mass norm on real homology goes back to Federer
\cite{Federer1974}.  In codimension-one, it may be regarded as the
effective or homogenized area functional associated with the Riemannian
metric.  A calibrated hypersurface realizes the stable norm of its
homology class.

In a recent work \cite{Nguyen2026}, the author considers, for an integer class $r\in H_{n-1}(\mathbb T^n, \mathbb Z)$, the min-max excess quantity:
\[\Delta W_r=\omega(r)-S(r),\]
where $\omega(r)$ is the Almgren-Pitts width defined using sweep-out inside $r$ and $S(r)$ is the stable norm. The author proves that for a totally irrational class $\alpha$, $\Delta W_r \to 0$ as $r/S(r)\to \alpha/S(\alpha)$ if and only if there is a foliation by minimizing hypersurfaces in the direction $\alpha$. Throughout this paper, we will use the more standard notation, $\|r\|_{\st}$, to denote the stable norm.

The regularity of the stable norm is closely connected with the geometry of
minimizing laminations. Auer and Bangert \cite{AuerBangert2006} proved that, at a nonzero class
$\alpha$, the restriction of the codimension-one stable norm to the
smallest rational subspace containing $\alpha$ is differentiable. Junginger-Gestrich related differentiability in
the remaining directions to the absence of gaps in the corresponding
minimizing lamination \cite{JungingerGestrich2007}.  Combining these
results with a compactness argument for rationally dependent classes, we prove in Appendix \ref{app:differentiability} that, on a torus, Bangert's condition is equivalent to the fact that the stable norm on $H_{n-1}(\T^n,\R)$ is differentiable at every nonzero class.
Thus Bangert's problem can also be formulated as a rigidity question for
the stable norm:
\[
\|\cdot\|_{\st,g}
\text{ is differentiable on }
H_{n-1}(\T^n,\R)\setminus\{0\}
\quad\stackrel{?}{\Longrightarrow}\quad
g\text{ is flat}.
\]

Bangert \cite{Bangert1994} constructed a metric on $\mathbb T^n$ for which all but one
primitive codimension-one homology class admit foliations by closed
minimal hypersurfaces (see also
\cite{Nguyen2026}).  The construction below is motivated in part by this
example and gives a negative answer to Bangert's question.
\subsection{Main results}

Consider for $n\geq 3$ the torus
\[
  \T^n=\T^{n-1}_x\times S^1_t
\]
with local coordinates $(x_1, \ldots, x_{n-1}, t)$ and consider Riemannian metrics of the form
\begin{equation}\label{eq:generalmetric-intro}
  g=q(t)^2dt^2+h_t,
  \qquad \det h_t=1,
\end{equation}
where $q:S^1\to(0,\infty)$ is smooth and $h_t$ is a smooth periodic family
of translationally invariant metrics on $\T^{n-1}$.  Geometrically,
$t\mapsto h_t$ is a loop of marked flat, unit-volume $(n-1)$-tori.  We will see that the
determinant condition preserves the volume of the horizontal fibers, making them minimal with respect to the metric $g$.

\begin{theorem}
\label{thm:generalformula-intro}
Let $g$ be a metric of the form \eqref{eq:generalmetric-intro}.  For
$(z,c)\in\R^{n-1}\times \mathbb R$,
consider the homology class $r_{z,c}\in H_{n-1}(\T^n,\R)$ defined by the Poincar\'e dual: 
\[
  \PD(r_{z,c})=\sum_{i=1}^{n-1}z_i[dx^i]+c[dt].
\]
Its stable norm is given by
\begin{equation}\label{eq:mainformula-intro}
  \|r_{z,c}\|_{\st,g}
  =\sqrt{\mathcal A_g(z)^2+c^2}, \quad \text{where} \quad \mathcal A_g(z)
  :=\int_0^1q(t)\sqrt{z^Th_t^{-1}z}\,dt.
\end{equation}
Moreover, for every primitive $(z,c)\in\Z^{n-1}\times\Z$, the class
$r_{z,c}$ is represented by a smooth fibration of $\T^n$ by connected,
embedded, calibrated $(n-1)$-tori.  For $z=0$ these are the horizontal
fibers; for $z\ne0$ they are constant-angle suspensions constructed in
Section \ref{sec:formula}.
\end{theorem}

The quantity $\mathcal A_g(z)$ is the area of the canonical vertical
hypersurface obtained by sweeping the flat subtorus dual to $z$ once around
the base circle. 

The theorem has the following immediate consequence.

\begin{corollary}
\label{cor:bangert-intro}
For every $n\ge3$ there exist smooth nonflat metrics on $\T^n$ for which
every primitive class in $H_{n-1}(\T^n,\Z)$ admits a foliation by connected,
homologically area-minimizing $(n-1)$-tori.  Consequently, the answer to Bangert's question is negative in every
dimension $n\ge3$.
\end{corollary}

 In his Scalar curvature notes \cite{Gromov2019}, Gromov posed a $C^0$-stability question for the rigidity of scalar curvature on $\mathbb T^n$:
given smooth metrics $g_i$ on $\mathbb T^n$ satisfying
\[
\operatorname{Sc}(g_i)\geq -\varepsilon_i,
\qquad
\varepsilon_i\longrightarrow0,
\]
and converging uniformly to a continuous Riemannian metric $g$, he asked
whether $(\mathbb T^n,g)$ must be isometric to a flat torus.  His no-gap argument yields, for every
primitive codimension-one class, a foliation by $g$-minimal
hypersurfaces, but does not establish that the leaves are totally
geodesic.  Thus an affirmative answer to Bangert's question would have
implied flatness of the limit.  Corollary~\ref{cor:bangert-intro} shows
that the foliation conclusion alone is insufficient.  The corresponding
$C^0$ rigidity problem was subsequently resolved by Burkhardt-Guim
using regularizing Ricci flow \cite{BurkhardtGuim2019}.

One could ask the following substantially stronger rigidity question: whether the codimension-one stable norm determines the metric. This can be viewed as a higher-dimensional
analogue of the marked length spectrum for surfaces. A basic rigidity
question is how much of the geometry is determined by this marked
minimizing data. For closed negatively curved surfaces, Croke and Otal \cite{Croke1990,Otal1990}
independently proved that the marked length spectrum-that is, the
length of the closed geodesic in each nontrivial free homotopy
class-determines the metric up to an isometry isotopic to the identity. This rigidity fails in general without a
negative-curvature assumption, but it remains true when one of the
metrics is a flat metric on \(\T^2\), see \cite{Bangert1994}. Therefore, we can ask the following higher-dimensional rigidity question:

\begin{?}
Let $g_0$ be a
flat metric on the marked torus $\T^n=\R^n/\Z^n$.  If
\begin{equation}\label{eq:stronger-rigidity-question}
  \|\cdot\|_{\st,g}=\|\cdot\|_{\st,g_0}
  \qquad\text{on }H_{n-1}(\T^n,\R),
\end{equation}
must $g$ itself be flat?  
\end{?}

 Our second result shows that when $n=3$, the codimension-one stable norm alone remains nonrigid even when its entire unit ball is exactly Euclidean.

\begin{theorem}
\label{thm:cubic-intro}
There exists a smooth two-parameter family of nonflat metrics
\[
  \{g_{\lambda,s}:\lambda>1,\ 0\le s<1\}
\]
on $\T^3$ such that
\[
  \|r_{a,b,c}\|_{\st,g_{\lambda,s}}=\sqrt{a^2+b^2+c^2}
\]
for every $(a,b,c)\in H_2(\T^3,\R)\cong\R^3$.  Thus every
$g_{\lambda,s}$ has exactly the codimension-one stable norm of the unit
cubic flat torus and they are pairwise non-isometric on $(\lambda, s)\in (1, \infty)\times [0,1)$.For each fixed $\lambda$, the metrics have the same volume and are pairwise non-isometric.     For each fixed $s$, as $\lambda\to 1^+$, they converge in $C^\infty$ to the cubic flat metric after pullback by diffeomorphisms isotopic to the identity. 
\end{theorem}

Theorem~\ref{thm:cubic-intro} shows that, on $\T^3$, the codimension-one
stable norm alone does not determine flatness, even locally.  Our third
main result shows that the flat volume is nevertheless rigid under a
free isometric circle action.  In fact, the cubic flat metric maximizes
volume among metrics with that symmetry and the prescribed stable norm.

\Needspace{13\baselineskip}
\begin{theorem}
\label{thm:circle-rigidity-intro}
Let $g_0$ be the unit cubic flat metric on the standard marked torus
$\T^3$.  Suppose that $g$ is a smooth Riemannian metric admitting a free
isometric $S^1$-action and that
\[
  \|\cdot\|_{\st,g}=\|\cdot\|_{\st,g_0}
  \qquad\text{on }H_2(\T^3,\R).
\]
Then
\[
  \Vol(\T^3,g)\le1.
\]
Equality holds if and only if there exists a diffeomorphism $\phi$,
isotopic to the identity, such that $\phi^*g=g_0$.
\end{theorem}

Every metric in \eqref{eq:generalmetric-intro} on $\T^3$ is invariant
under a free $\T^2$-action.  Theorem~\ref{thm:circle-rigidity-intro}
requires only one circle symmetry and no closeness to a flat metric.
For an invariant metric $g=h+f^2(dt+\beta)^2$, circle averaging of the
stable-norm cell formula \cite{ChambolleGoldmanNovaga2014} identifies
the horizontal restriction with the stable norm of the weighted quotient
metric $f^2h$ on $\T^2$.  Bangert's two-dimensional rigidity theorem
\cite{Bangert1994} makes that quotient flat. The reduced cell formula
and Cauchy-Schwarz then yield the volume bound and its equality case.
The proof is given in Section~\ref{sec:circle-rigidity}.

The flexibility in our counterexamples is that the minimizing foliation in a
mixed class is allowed to depend nonlinearly on that class: its lifted leaves
are $z\cdot x+c\,\tau_z(t)=\theta$, rather than the standard affine
hyperplanes $z\cdot x+ct=\theta$.  This distinction is essential.  
If we strictly require that the foliations on $\mathbb T^n$ are all by affine hyperplanes, then we do have flat rigidity. The proof we give in Appendix \ref{app:affine} follows from a characterization of Hangan \cite{Hangan1996} on a metric on $\mathbb R^n$ for which every hyperplane is minimal. 

\subsection{Open questions}
\label{sec:open}

The preceding results leave two natural questions.  The first asks
whether volume restores rigidity of the flat metric without a symmetry
assumption.  Osuna proved that,
under the unit-volume normalization, equality of both the $H_1$ and
$H_{n-1}$ stable norms with those of a flat torus forces isometry with
that torus \cite{Osuna2010}.  This motivates the following question,
which imposes no condition on the $H_1$ stable norm.

\begin{question}
\label{open:volume-rigidity}
Let $g$ be a smooth Riemannian metric on the standard marked torus $\T^n$.
Suppose that
\[
  \|r_{a_1,a_2,\ldots, a_n}\|_{\st,g}=\sqrt{a_1^2+a_2^2+\cdots +a_n^2}
\]
for every $(a_1, \ldots, a_n)\in H_{n-1}(\T^n,\R)\cong\R^n$, and that
\[
  \operatorname{Vol}(\T^n,g)=1.
\]
Must $g$ be isometric to a flat metric by a diffeomorphism
isotopic to the identity?
\end{question}

The metrics in Theorem~\ref{thm:cubic-intro} satisfy
\[
  \Vol(\T^3,g_{\lambda,s})=\kappa_\lambda^{-1}<1
  \qquad (\lambda>1).
\]
Thus their nonuniqueness at a fixed subunit volume does not answer
Question~\ref{open:volume-rigidity} negatively.  Marques--Neves--Sun
\cite{MarquesNevesSun2026} prove local rigidity of the flat metric under
the prescribed stable norm and equal-volume assumptions, without a
symmetry hypothesis.  Theorem~\ref{thm:circle-rigidity-intro} answers the
question globally on $\T^3$ under a free isometric circle action and
proves the stronger volume bound $\Vol(\T^3,g)\le1$, with rigidity of
equality. Since this class of metrics contains cohomogeneity-one metrics defined by \eqref{eq:generalmetric-intro}, one would not hope to find counter-example of Question \ref{open:volume-rigidity} in this class. The unrestricted global question remains open.

The second question concerns higher dimensions.  The rotational
averaging in Section~\ref{sec:cubic} produces the Euclidean stable norm
using a two-dimensional horizontal factor.  The same rotation argument
does not directly yield the corresponding identity in higher horizontal
dimensions.

\begin{question}
\label{q:higher-cubic}
For $n\ge4$, does there exist a smooth nonflat metric on the standard torus
$\T^n$ whose stable norm on $H_{n-1}(\T^n,\R)$ is exactly
\[
  (a_1,\ldots,a_n)\longmapsto
  \sqrt{a_1^2+\cdots+a_n^2}
\]
under the standard integral identification?
\end{question}
\subsection{Organization of the paper}
Section~\ref{sec:prelim} recalls the definitions of stable norms and
calibrations.  Section~\ref{sec:formula} proves the exact stable-norm
formula, constructs the calibrated foliations, and derives the
counterexamples to Bangert's question.  Section~\ref{sec:cubic} constructs
the two-parameter family on $\T^3$ with the exact Euclidean stable norm
and shows that the stable-norm-and-volume data are not locally injective
near the cubic flat metric.  Section~\ref{sec:circle-rigidity} proves
the global volume comparison and equality rigidity under a free
isometric circle action.  Appendix~\ref{app:differentiability} proves
the equivalence between Bangert's foliation hypothesis and
differentiability of the stable norm away from the origin, and
Appendix~\ref{app:affine} proves rigidity under the affine foliation
hypothesis.

\medskip
\noindent\textbf{Acknowledgement}. I am deeply grateful to my advisor, Andr\'{e} Neves, for many helpful discussions, guidance, and support.

\medskip
\noindent\textbf{Note added.}
While preparing this manuscript, we learned of the recent independent
work of Marques, Neves, and Sun \cite{MarquesNevesSun2026}.  Their Theorem~5.1 gives, in an equivalent cofactor formulation,
the stable-norm formula for the cohomogeneity-one class considered here
and constructs foliations by area-minimizing hypersurfaces in every
primitive codimension-one homology class.  Their Theorems~1.4 and~1.5
independently provide nonflat examples giving a negative answer to
Bangert's question.  They establish local rigidity near a flat
metric under the additional assumption of equal volume.  The overlapping
results in the two works were obtained independently.

\section{Stable norm, Poincar\'e duality, and calibrations}
\label{sec:prelim}

Let $(M^n,g)$ be a closed oriented Riemannian manifold. The stable norm of
$r\in H_k(M,\R)$ can be defined by
\begin{equation}\label{eq:stable-definition}
  \|r\|_{\st,g}
  :=\inf\{\Mass_g(T): T\text{ is a closed real }k\text{-current and }[T]=r\}.
\end{equation}
Equivalently, one may use real Lipschitz cycles.  We will only use
$k=n-1$, where minimizers are particularly well suited to calibration
methods, see \cite{Federer1974} and \cite{AuerBangert2006}.

A smooth $(n-1)$-form $\Omega$ has comass at most one if
\[
  |\Omega(\xi)|\le1
\]
for every unit simple $(n-1)$-vector $\xi$.  If $d\Omega=0$ and an oriented
hypersurface $\Sigma$ satisfies
\[
  \Omega|_\Sigma=dA_\Sigma,
\]
then $\Omega$ calibrates $\Sigma$.  The standard calibration argument gives
the following lemma. See also \cite{HarveyLawson1982}.

\begin{lemma}
\label{lem:calibration}
If a closed comass-one $(n-1)$-form $\Omega$ calibrates an oriented closed
hypersurface $\Sigma$, then
\[
  \Mass_g(T)\ge \operatorname{Area}_g(\Sigma)
\]
for every closed real $(n-1)$-current $T$ homologous to $\Sigma$.  Hence
$\Sigma$ realizes the stable norm of its homology class.
\end{lemma}

\begin{proof}
The comass inequality gives $\Mass_g(T)\ge T(\Omega)$.  Since $d\Omega=0$
and $[T]=[\Sigma]$, it follows from Stokes' theorem that
\[
  T(\Omega)=\Sigma(\Omega)=\operatorname{Area}_g(\Sigma).
\]
\end{proof}

For the torus
\[
  \T^n=\R^n/\Z^n=\T^{n-1}_x\times S^1_t,
\]
we orient by
\[
  dx^1\wedge\cdots\wedge dx^{n-1}\wedge dt.
\]
Poincar\'e duality identifies $H_{n-1}(\T^n,\R)$ with $H^1(\T^n,\R)$.
We write $r_{z,c}$ for the class characterized by
\begin{equation}\label{eq:PD-coordinates}
  \PD(r_{z,c})=\sum_{i=1}^{n-1} z_i[dx^i]+c[dt],
  \qquad (z,c)\in\R^{n-1}\times\R.
\end{equation}
The signs of the coordinate hypersurface basis are chosen consistently with
\eqref{eq:PD-coordinates}.

We have the following identity
\begin{equation}\label{eq:div-calibration}
  d(\iota_\nu dV_g)
  =\mathcal L_\nu dV_g
  =(\operatorname{div}_g\nu)dV_g
\end{equation}
for a unit vector field $\nu$. Thus a cooriented hypersurface foliation with
divergence-free unit normal is automatically calibrated leafwise.

\section{The exact stable norm and counterexamples to Bangert's conjecture}
\label{sec:formula}

In this section, we prove Theorem \ref{thm:generalformula-intro} and Corollary \ref{cor:bangert-intro}. Recall
\[
  \T^n=\T^{n-1}_x\times S^1_t
\]
with coordinates $(x_1,\ldots, x_{n-1},t)$ and consider metrics
\begin{equation}\label{eq:generalmetric}
  g=q(t)^2dt^2+h_t,
  \qquad \det h_t=1,
\end{equation}
where $q:S^1\to(0,\infty)$ is smooth and $h_t$ is a smooth periodic
family of constant-coefficient metrics on the horizontal torus
$\T^{n-1}\times\{t\}$. Geometrically, the metric is a loop of marked
flat, unit-volume $(n-1)$-tori over a circle. The base arc-length element is
$q(t)dt$. Its volume form is
\begin{equation}\label{eq:volumeform}
  dV_g=q(t)\,dx_1\wedge\cdots\wedge dx_{n-1}\wedge dt.
\end{equation}

The vertical unit vector field is
\begin{equation}\label{eq:verticalunit}
  \nu_0=q(t)^{-1}\partial_t.
\end{equation}
The horizontal fibers $F_t=\T^{n-1}\times\{t\}$ have mean curvature
\begin{align}\label{eq:horizontal-mean-curvature}
  H_{F_t}=\operatorname{div}_g \nu_0
  &=\frac{1}{q(t)\sqrt{\det h_t}}
    \partial_t\!\left(q(t)\sqrt{\det h_t}\,\frac1{q(t)}\right) \\
  &=\frac1{q(t)}\frac{d}{dt}\log\sqrt{\det h_t}=0.
\end{align}
Thus all horizontal fibers are minimal. 

As in \eqref{eq:PD-coordinates}, write
\[
  \PD(r_{z,c})=\sum_{i=1}^{n-1}z_i[dx_i]+c[dt],
  \qquad (z,c)\in\R^{n-1}\times\R.
\]
We consider two cases: $z=0$ and $z\ne0$.

\subsection{The horizontal class}

When $z=0$, the primitive classes are $r_{0,\pm1}$, represented by the
horizontal foliation $\{F_t\}_{t\in S^1}$. Since
$\operatorname{div}_g \nu_0=0$, the form $\iota_{\nu_0}dV_g$ is closed, has comass
one, and restricts to the volume form of every $F_t$. It therefore
calibrates the horizontal foliation. Since $\det h_t=1$, each horizontal
fiber has area one, and hence
\begin{equation}\label{eq:horizontal-primitive-norm}
  \|r_{0,\pm1}\|_{\st,g}=1.
\end{equation}

\subsection{The classes with nonzero horizontal component}

For $z\in\R^{n-1}\setminus\{0\}$, define
\begin{equation}\label{eq:sweptarea}
  \mathcal A_g(z)
  :=\int_0^1 q(t)\sqrt{z^Th_t^{-1}z}\,dt.
\end{equation}
Geometrically, for integral $z\neq0$, $\mathcal A_g(z)$ is the area
of the hypersurface obtained by sweeping the horizontal subtorus dual
to $z$ once around the base circle.

On the universal cover $\R^{n-1}\times\R$, define
\begin{equation}\label{eq:utilde}
  \widetilde u_{z,c}(x,t)
  :=z\cdot x
  +\frac{c}{\mathcal A_g(z)}
   \int_0^t q(s)\sqrt{z^Th_s^{-1}z}\,ds.
\end{equation}
For $(z,c)\in \mathbb Z^{n-1}\times \mathbb Z$, it follows from periodicity that for every
$(k,\ell)\in\Z^{n-1}\times\Z$ we have
\[
  \widetilde u_{z,c}(x+k,t+\ell)
  -\widetilde u_{z,c}(x,t)
  =z\cdot k+c\ell\in\Z.
\]
Hence $\widetilde u_{z,c}$ descends modulo $\Z$ to a circle-valued
function
\[
  u_{z,c}:\T^n\longrightarrow S^1.
\]
Since $z\ne0$, its differential never vanishes, so $u_{z,c}$ is a
submersion. Its fibers
\begin{equation}\label{eq:helicalleaves}
  \Sigma_{z,c,\theta}
  :=u_{z,c}^{-1}(\theta)
  =\left\{(x,t):
  z\cdot x+\frac{c}{\mathcal A_g(z)}
  \int_0^t q(s)\sqrt{z^Th_s^{-1}z}\,ds
  =\theta\pmod1\right\}
\end{equation}
belong to the class $r_{z,c}$ and form a smooth foliation of $\T^n$.

\begin{lemma}
\label{lem:fibertopology}
If $(z,c)\in\Z^n$ is primitive and $z\ne0$, every fiber of $u_{z,c}$ is a
connected embedded $(n-1)$-torus representing $r_{z,c}$.
\end{lemma}

\begin{proof}
Set
\[
  \tau_z(t)
  :=\frac1{\mathcal A_g(z)}
    \int_0^t q(s)\sqrt{z^Th_s^{-1}z}\,ds.
\]
Then $\tau_z'(t)>0$ and $\tau_z(t+1)=\tau_z(t)+1$, so
$(x,t)\mapsto(x,\tau_z(t))$ descends to a diffeomorphism of $\T^n$.
Under this diffeomorphism, $u_{z,c}$ becomes the linear function
\[
  (x,\theta)\longmapsto z\cdot x+c\theta\pmod1.
\]
Its induced map on fundamental groups is
\[
  \Z^n\longrightarrow\Z,
  \qquad (k,\ell)\longmapsto z\cdot k+c\ell.
\]
Primitivity is equivalent to surjectivity, and the kernel is a rank $(n-1)$
lattice. Hence every fiber is connected and diffeomorphic to $\T^{n-1}$.
Moreover,
\[
  [du_{z,c}]
  =\sum_{i=1}^{n-1}z_i[dx_i]+c[dt],
\]
so an oriented fiber represents $r_{z,c}$.
\end{proof}

Differentiating \eqref{eq:utilde} gives
\begin{equation}\label{eq:du}
  du_{z,c}
  =\sum_{i=1}^{n-1}z_i\,dx^i
   +\frac{c}{\mathcal A_g(z)}
    q(t)\sqrt{z^Th_t^{-1}z}\,dt.
\end{equation}
Its unit normal is then given by
\begin{equation}\label{eq:unitnormal}
  \nu_{z,c}
  =\frac{1}{\sqrt{\mathcal A_g(z)^2+c^2}}
  \left(
    \mathcal A_g(z)\frac{h_t^{-1}z}{\sqrt{z^Th_t^{-1}z}}
    +\frac{c}{q(t)}\partial_t
  \right).
\end{equation}
The horizontal components of $\nu_{z,c}$ depend only on $t$, whereas
\[
  q(t)\nu_{z,c}^t
  =\frac{c}{\sqrt{\mathcal A_g(z)^2+c^2}}
\]
is constant. Since $\det h_t=1$, the volume density of $g$ is $q(t)$, and
therefore
\[
  \operatorname{div}_g\nu_{z,c}
  =\frac1{q(t)}\left[
    \sum_{i=1}^{n-1}\partial_{x^i}\bigl(q(t)\nu_{z,c}^i\bigr)
    +\partial_t\bigl(q(t)\nu_{z,c}^t\bigr)
  \right]
  =0.
\]
Consequently, the closed comass-one form
$\iota_{\nu_{z,c}}dV_g$ calibrates every fiber.

It remains only to compute the area of a fiber. Horizontal translations
are isometries of $g$ and act transitively on the fibers of $u_{z,c}$, so
all fibers have the same area.
Since $g^{-1}=q(t)^{-2}\partial_t^2+h_t^{-1}$, it follows from \eqref{eq:du} that 

\[|du_{z,c}|^2_g=z^Th_t^{-1}z+q(t)^{-2}\left(\dfrac{c}{\mathcal A_g(z)}q(t)\sqrt{z^Th_t^{-1}z}\right)^2=z^Th_t^{-1}z\dfrac{\mathcal A_g(z)^2+c^2}{\mathcal A_g(z)^2}.\]
For any $0\leq \theta_0\leq 1$, the coarea formula then gives
\begin{align}\label{eq:fiber-area}
  \operatorname{Area}_g(\Sigma_{z,c,\theta_0})&=\int_{S^1}\operatorname{Area}_g(\Sigma_{z,c,\theta})d\theta\\
  &=\int_{\T^n}|du_{z,c}|_g\,dV_g \\
  &=\frac{\sqrt{\mathcal A_g(z)^2+c^2}}{\mathcal A_g(z)}
    \int_0^1q(t)\sqrt{z^Th_t^{-1}z}\,dt \\
  &=\sqrt{\mathcal A_g(z)^2+c^2}.
\end{align}
Because the fiber is calibrated, its area is the least mass in its homology
class. In particular, we have for every metric in \eqref{eq:generalmetric},
\begin{equation}\label{eq:mainformula}
  \|r_{z,c}\|_{\st,g}
  =\sqrt{\mathcal A_g(z)^2+c^2}
\end{equation}
for every primitive class $r_{z,c}\in H_{n-1}(\mathbb T^n , \mathbb Z)$.

The argument proves the formula \eqref{eq:mainformula-intro} for primitive integral
classes.  If $(z,c)=k(z_0,c_0)$ with $k\in\mathbb N$ and
$(z_0,c_0)$ primitive, then the absolute homogeneity of the stable norm and
of $\mathcal A_g$ gives
\[
\|r_{z,c}\|_{\mathrm{st},g}
=
k\|r_{z_0,c_0}\|_{\mathrm{st},g}
=
\sqrt{\mathcal A_g(z)^2+c^2}.
\]
Thus the formula holds for every integral class.  It then holds for
rational classes by rescaling.  Finally, rational classes are dense in
$H_{n-1}(\mathbb T^n,\mathbb R)$, and both the stable norm and
\[
(z,c)\longmapsto\sqrt{\mathcal A_g(z)^2+c^2}
\]
are continuous.  Passing to the limit proves the formula for every real
class.

\subsection{Counterexamples to Bangert's conjecture}

We now derive Corollary~\ref{cor:bangert-intro} directly from
Theorem~\ref{thm:generalformula-intro}.  Let
\[
  g=q(t)^2dt^2+h_t
\]
be a metric of the form \eqref{eq:generalmetric}, and assume that the
periodic family $t\mapsto h_t$ is nonconstant.  Theorem~\ref{thm:generalformula-intro}
gives, in every primitive class in $H_{n-1}(\T^n,\Z)$, a foliation by
connected calibrated homologically area-minimizing tori.

It remains only to verify that $g$ is nonflat.  Since the family
$t\mapsto h_t$ is nonconstant, the shape operator 
\[
  B(t)
  =\frac{1}{2q(t)}
    h_t^{-1}\dot h_t
\]
of some
horizontal slice $F_{t_0}=\T^{n-1}\times \{t_0\}$ is nonzero.  The condition
$\det h_t\equiv1$ implies that $B(t_0)$ is trace-free, hence it has
principal curvatures $\kappa_+>0$ and $\kappa_-<0$.  If $E_+$ and
$E_-$ are corresponding orthonormal principal directions, then the
intrinsic flatness of $F_{t_0}$ and the Gauss equation give
\[
K_g(E_+,E_-)=-\kappa_+\kappa_->0.
\]
Thus $g$ has nonzero sectional curvature and is therefore nonflat.
This proves Corollary~\ref{cor:bangert-intro}.

\section{The Euclidean stable norm in dimension three}
\label{sec:cubic}

We now specialize the construction of Section~\ref{sec:formula} to
$\T^3=\T^2\times S^1$.  The aim is to choose a nonconstant unit-volume
loop of horizontal metrics for which the swept area norm $\mathcal A_g$ is
exactly Euclidean. Recall that $g=q(t)^2dt^2+h_t$. Since $\det h_t=1$, it can be orthogonally diagonalized as 
\[h_t=R_{\theta(t)}^T\operatorname{diag}(\alpha(t), \beta(t))R_{\theta(t)},\]
where $R_\theta$ is the rotation matrix by angle $\theta$.
We try to look for examples where $\alpha, \beta$ are constant. Set
$\alpha(t)=\lambda^2$, $\beta(t)=\lambda^{-2}$ where $\lambda>0$ is a fixed constant and set
\[
  D_\lambda=\operatorname{diag}(\lambda^2,\lambda^{-2}).
\]
For further simplification, let $\theta(t)=2\pi t$ and set
\begin{equation}\label{eq:Hlambda}
  h_\lambda(t)=R_{2\pi t}^{T}D_\lambda R_{2\pi t}.
\end{equation}
We then have that $\det h_\lambda(t)=1$.

For $z\in\R^2$, the function
\[
  z\longmapsto
  \int_0^1\sqrt{z^Th_\lambda(t)^{-1}z}\,dt
\]
is invariant under planar rotations and is positively homogeneous.
Consequently, it is a constant multiple of the Euclidean norm.  We write
\begin{equation}\label{eq:rotational-average}
  \int_0^1\sqrt{z^Th_\lambda(t)^{-1}z}\,dt
  =\kappa_\lambda|z|,
\end{equation}
where evaluation at $z=e_1$ gives
\begin{equation}\label{eq:kappa-cubic}
  \kappa_\lambda
  =\int_0^1
    \sqrt{\lambda^{-2}\cos^2(2\pi t)
          +\lambda^2\sin^2(2\pi t)}\,dt.
\end{equation}

The next theorem gives a general condition on the vertical function $q(t)$ under
which the codimension-one stable norm is Euclidean.

\begin{theorem}\label{thm:general-q-cubic}
Fix $\lambda>1$, and let $q:S^1\to(0,\infty)$ be a smooth function
satisfying
\begin{equation}\label{eq:q-half-period-condition}
  q(t)+q\left(t+\frac12\right)=\frac{2}{\kappa_\lambda}
  \qquad\text{for every }t\in S^1.
\end{equation}
Then the metric
\begin{equation}\label{eq:g-lambda-q}
  g_{\lambda,q}=q(t)^2dt^2+h_\lambda(t)
\end{equation}
on $\T^3$ is nonflat and satisfies
\begin{equation}\label{eq:cubicnorm-general-q}
  \|r_{a,b,c}\|_{\st,g_{\lambda,q}}
  =\sqrt{a^2+b^2+c^2}
\end{equation}
for every $(a,b,c)\in H_2(\T^3,\R)\cong\R^3$.  Moreover,
\begin{equation}\label{eq:volume-general-q}
  \Vol(\T^3,g_{\lambda,q})=\kappa_\lambda^{-1}.
\end{equation}
\end{theorem}

\begin{proof}
It follows from the definition that $h_\lambda$ satisfies the periodicity condition:
\[h_\lambda\left(t+\dfrac{1}{2}\right)=h_\lambda(t)\]
For $z\in \mathbb R^2$, we compute the function $\mathcal A_{g_{\lambda, q}}$ as in Section \ref{sec:formula}, using a change of variable as follows:
\begin{align*}
  \mathcal A_{g_{\lambda,q}}(z)
  &=\int_0^1q(t)\sqrt{z^T h_\lambda(t)^{-1}z}\,dt\\
  &=\int_0^{1/2}
    \left[q(t)+q\left(t+\frac12\right)\right]\sqrt{z^Th_\lambda(t)^{-1}z}\,dt\\
  &=\frac{2}{\kappa_\lambda}\int_0^{1/2}\sqrt{z^Th_\lambda(t)^{-1}z}\,dt\\
  &=\frac{1}{\kappa_\lambda}\int_0^1\sqrt{z^Th_\lambda(t)^{-1}z}\,dt
  =|z|,
\end{align*}
where the last equality follows from \eqref{eq:rotational-average}.
Theorem~\ref{thm:generalformula-intro} therefore gives
\eqref{eq:cubicnorm-general-q}.

Since $\det h_\lambda(t)=1$, the volume form on $\mathbb T^3$ is
$dV_g=q(t)\,dx\,dy\,dt$.  Integrating
\eqref{eq:q-half-period-condition} over $[0,\frac12]$ with a change of variable gives
\[
  \int_0^1q(t)\,dt=\kappa_\lambda^{-1},
\]
which proves \eqref{eq:volume-general-q}.  Finally, the map
$t\mapsto h_\lambda(t)$ is non-constant when $\lambda>1$, so the nonflatness of the constructed metrics
follows from the argument in Section~\ref{sec:formula}.
\end{proof}

\begin{remark}
    Condition~\eqref{eq:q-half-period-condition} is equivalent to
\[
q(t)
=
\kappa_\lambda^{-1}(1+\psi(t)),
\qquad
\psi\left(t+\frac12\right)=-\psi(t),
\qquad
\|\psi\|_{C^0}<1.
\]
These functions form an open subset of the infinite-dimensional
closed linear subspace of smooth functions satisfying
$\psi(t+\tfrac12)=-\psi(t)$.
\end{remark}

By specializing $q(t)$ to some special form, we can state and prove Theorem~\ref{thm:cubic-intro} in a more precise form as follows. 

\begin{corollary}\label{cor:sinusoidal-family}
For $(\lambda, s)\in (1, \infty)\times [0,1)$, define
\begin{equation}\label{eq:qs-family}
  q_{\lambda,s}(t)
  =\kappa_\lambda^{-1}\bigl(1+s\sin(2\pi t)\bigr)
\end{equation}
and
\begin{equation}\label{eq:gs-family}
  g_{\lambda,s}=q_{\lambda,s}(t)^2dt^2+h_\lambda(t).
\end{equation}
Then every $g_{\lambda,s}$ is nonflat and has the Euclidean
codimension-one stable norm,
\begin{equation}\label{eq:cubicnorm}
  \|r_{a,b,c}\|_{\st,g_{\lambda,s}}
  =\sqrt{a^2+b^2+c^2}.
\end{equation}
For each fixed $\lambda>1$, the metrics $g_{\lambda,s}$ have the same
volume $\kappa_\lambda^{-1}$ and are pairwise non-isometric.  For each
fixed $s<1$, as $\lambda\downarrow1$ they converge, after pullback by
diffeomorphisms isotopic to the identity, to the unit Euclidean flat metric
in the $C^\infty$ topology.
\end{corollary}

\begin{proof}
The function $q(t)=\kappa_\lambda^{-1}(1+s\sin(2\pi t))$ obviously satisfies the condition \eqref{eq:q-half-period-condition}; therefore the stable-norm, volume, and
nonflatness statements follow from Theorem~\ref{thm:general-q-cubic}.

It remains to distinguish the metrics for different values of $s$. We can do that by directly computing their scalar curvatures.

The shape operator of the horizontal two-tori can be computed as:
\[
  B_{\lambda,s}(t)
  =\frac{1}{2q_{\lambda,s}}
    h_\lambda(t)^{-1}\dot h_\lambda(t).
\]
Using the scalar Gauss-Riccati identity
$\Sc(g_{\lambda,s})=-|B_{\lambda,s}|^2$, one obtains
\begin{equation}\label{eq:scalar-family}
  \Sc(g_{\lambda,s})(t)
  =-
  \frac{2\pi^2\kappa_\lambda^2
  (\lambda^2-\lambda^{-2})^2}
  {(1+s\sin(2\pi t))^2}.
\end{equation}
Taking the infimum over $t$ we have 
\[\min \Sc(g_{\lambda, s})=-\dfrac{2\pi^2\kappa_\lambda^2(\lambda^2-\lambda^{-2})^2}{(1-s)^2}.\]
It follows that $g_{\lambda,s}$ and
$g_{\lambda,s'}$ are not isometric when $s\ne s'$ and $\lambda\neq 1$. Metrics corresponding to two different values $\lambda, \lambda'>1$ are distinguished by their volumes. 

Finally,
\[
  h_\lambda(t)\longrightarrow I,
  \qquad
  \kappa_\lambda\longrightarrow1
\]
in $C^\infty$ as $\lambda\downarrow1$.  Therefore
\[
  g_{\lambda,s}
  \longrightarrow
  (1+s\sin(2\pi t))^2dt^2+dx^2+dy^2.
\]
The limiting metric is the pullback of the unit Euclidean metric under the
circle reparametrization
\[
  \tau_s(t)=\int_0^t\bigl(1+s\sin(2\pi u)\bigr)\,du.
\]
Since this reparametrization is isotopic to the identity, the last conclusion follows. This completes the proof. 
\end{proof}

\section{Stable norm and volume rigidity under circle symmetry}
\label{sec:circle-rigidity}

We prove Theorem~\ref{thm:circle-rigidity-intro} by first averaging the stable norm over the circle, and then invoking the
two-dimensional rigidity result. A free circle action on
$\mathbb T^3$ has a closed oriented surface $\Sigma$ as its
orbit space. Its homotopy exact sequence gives a surjection
$\mathbb Z^3\to\pi_1(\Sigma)$ and excludes $\Sigma=S^2$,
so $\Sigma=\mathbb T^2$.  The fiber subgroup is infinite cyclic, since
$\pi_2(\T^2)=0$.  Writing
$[u,v]=w^e$, with $w$ a fiber generator and $e$ the Euler
number, abelianity of $\pi_1(\mathbb T^3)$ forces $e=0$.
Hence this principal circle bundle is trivial.
In an equivariant product trivialization $\T^3=\T^2_x\times S^1_t$,
with the action of period one, every invariant metric has the form
\begin{equation}\label{eq:circle-invariant-metric}
  g=h(x)+f(x)^2\bigl(dt+\beta(x)\bigr)^2,
\end{equation}
where $h$ is a metric on $\T^2$, $f$ is a smooth positive function, and $\beta$ is a
smooth one-form.  Indeed, $f=|\partial_t|_g$ and
$dt+\beta=f^{-2}g(\partial_t,\cdot)$, while
$g-f^2(dt+\beta)^2$ descends to $h$.
For the weighted quotient metric $\widehat h=f^2h$, the volume form is
$dV_g=f^{-1}\,dA_{\widehat h}\,dt$.

In these product coordinates, we use the notation $r_{p,c}$ from
\eqref{eq:PD-coordinates}, and regard $p=(p_1,p_2)$ as the constant
one-form $p_1dx^1+p_2dx^2$ on $\T^2$.

\begin{lemma}
\label{lem:circle-cell}
For the metric \eqref{eq:circle-invariant-metric} and every
$(p,c)\in\R^2\times\R$,
\begin{equation}\label{eq:circle-reduced-cell}
  \|r_{p,c}\|_{\st,g}
  =\inf_{\psi\in C^\infty(\T^2)}
    \int_{\T^2}
    \sqrt{|p+d\psi-c\beta|_{\widehat h}^{\,2}+c^2f^{-4}}
    \,dA_{\widehat h}.
\end{equation}
\end{lemma}

\begin{proof}
For $\alpha=p+c\,dt$, the cohomological $L^1$ characterization gives
\[
  \|r_{p,c}\|_{\st,g}=\inf_{u\in C^\infty(\T^3)}
  \int_{\T^3}|\alpha+du|_g\,dV_g;
\]
see the duality formula in \cite{BangertKatz2004} and the
periodic cell formulation in
\cite{ChambolleGoldmanNovaga2014}.
Let $\rho_\theta(x,t)=(x,t+\theta)$ and
$\overline u(x)=\int_0^1u(x,t)\,dt$.  It follows that
\[\alpha+d\overline u=\int_0^1\rho_\theta^*(\alpha+du)\,d\theta.\]
The convexity and invariance of $g$ give
\[
  \int_{\T^3}|\alpha+d\overline u|_g\,dV_g
  \le \int_{\T^3}|\alpha+du|_g\,dV_g.
\]
Thus the infimum can be restricted to functions independent of $t$.
For such a function $\psi$, the splitting
$\alpha+d\psi=(p+d\psi-c\beta)+c(dt+\beta)$ yields
\[
  |\alpha+d\psi|_g\,dV_g
  =\sqrt{|p+d\psi-c\beta|_{\widehat h}^{\,2}+c^2f^{-4}}
    \,dA_{\widehat h}\,dt,
\]
which proves the formula without requiring a smooth minimizer.
\end{proof}

\begin{proof}[Proof of Theorem~\ref{thm:circle-rigidity-intro}]
In the product trivialization above, the induced change on integral
cohomology is unimodular.  Hence the prescribed cubic norm becomes
\begin{equation}\label{eq:circle-reference-ellipsoid}
  \|r_{p,c}\|_{\st,g}^{\,2}=p^TCp+2c\,\ell^Tp+e\,c^2,
  \qquad Q=\begin{pmatrix}C&\ell\\ \ell^T&e\end{pmatrix}>0,
  \qquad \det Q=1.
\end{equation}
By setting $c=0$ in \eqref{eq:circle-reduced-cell} we can identify
$\|r_{p,0}\|_{\st,g}$ with the stable norm of the weighted quotient
$(\T^2,\widehat h)$, expressed through Poincar\'e duality.
Since this norm is ellipsoidal, the two-dimensional stable norm rigidity result of Bangert
\cite{Bangert1994}, together with \cite{Innami1986},
implies that $\widehat h$ is flat.  After a base diffeomorphism
isotopic to the identity, write $\widehat h=H$ with $H$ constant.
The flat stable-norm formula gives
\[
  C=(\det H)H^{-1},\qquad \det C=\det H.
\]
Define $\gamma=-C^{-1}\ell$ and
$s=\sqrt{e-\ell^TC^{-1}\ell}>0$.
Direct determinant computation gives
$s\sqrt{\det H}=\sqrt{\det Q}=1$, while
\eqref{eq:circle-reference-ellipsoid} gives $\|r_{\gamma,1}\|_{\st,g}=s$.
Evaluating the cell formula at this real class, with $\gamma$
regarded as a constant one-form, yields
\begin{equation}\label{eq:circle-transverse-bound}
  \int_{\T^2}f^{-2}\,dA_H
  \le\inf_\psi\int_{\T^2}
  \sqrt{|d\psi-(\beta-\gamma)|_H^2+f^{-4}}\,dA_H=s.
\end{equation}
Since $\Vol(\T^2,H)=\sqrt{\det H}$, Cauchy--Schwarz gives
\begin{equation}\label{eq:circle-volume-chain}
  \begin{aligned}
    \Vol(\T^3,g)^2
    &=\left(\int_{\T^2}f^{-1}\,dA_H\right)^2\\
    &\le\sqrt{\det H}\int_{\T^2}f^{-2}\,dA_H
    \le s\sqrt{\det H}=1.
  \end{aligned}
\end{equation}

If $\Vol(\T^3,g)=1$, equality in \eqref{eq:circle-volume-chain}
forces $f$ to be constant and $\int_{\T^2}f^{-2}\,dA_H=s$.
Choose a smooth minimizing sequence $\psi_j$ in
\eqref{eq:circle-transverse-bound}, so that
\[
  I_j:=\int_{\T^2}
  \sqrt{|d\psi_j-(\beta-\gamma)|_H^2+f^{-4}}\,dA_H
  \longrightarrow s.
\]
Cauchy-Schwarz yields
\[
\begin{aligned}
&\left(
  \int_{\mathbb T^2}
  \bigl|d\psi_j-(\beta-\gamma)\bigr|_H\,dA_H
\right)^2
\\
&\quad\le
\left[
  \int_{\mathbb T^2}
  \left(
    \sqrt{\bigl|d\psi_j-(\beta-\gamma)\bigr|_H^2+f^{-4}}
    -f^{-2}
  \right)dA_H
\right]
\\
&\qquad\quad{}\times
\left[
  \int_{\mathbb T^2}
  \left(
    \sqrt{\bigl|d\psi_j-(\beta-\gamma)\bigr|_H^2+f^{-4}}
    +f^{-2}
  \right)dA_H
\right]
\\
&\quad=(I_j-s)(I_j+s)
=I_j^2-s^2
\longrightarrow0.
\end{aligned}
\]
Hence $d\psi_j\to\beta-\gamma$ in $L^1$, so $\beta-\gamma$ is
closed and has zero mean. Therefore, it is exact and we can write
$\beta-\gamma=dv$ for a smooth periodic function $v$.
The change of variable $t'=t+v(x)$ gives
\[
  g=f^{-2}H+f^2(dt'+\gamma)^2,
\]
which has constant coefficients.

Composing the resulting flat coordinate map with the inverse of its
induced linear torus automorphism, we obtain a diffeomorphism $F$
acting trivially on $H_1$ for which $F^*g=J$ is constant.  The
prescribed marked stable norm gives $\det J\,J^{-1}=\Id$, hence
$J=\Id$.  Since $F$ induces the identity on $\pi_1(\T^3)=\Z^3$,
it is homotopic to the identity and therefore isotopic to the identity
by \cite{Waldhausen1968}.  Thus $F^*g=g_0$.  The converse is immediate.
\end{proof}

\appendix

\section{Minimizing foliations and differentiability of the stable norm}
\label{app:differentiability}

We record the precise relation between Bangert's foliation condition and the
differentiability theory of the codimension-one stable norm.  The result is
a torus consequence of Auer-Bangert's differentiability theorem
\cite{AuerBangert2006}, Junginger-Gestrich's gap theorem
\cite{JungingerGestrich2007}, and a compactness argument that passes from
rational foliations to irrational directions.

 Let $\Gamma=H_1(\T^n,\Z)$, and let
$I:\Gamma\times H_{n-1}(\T^n,\R)\to\R$ be the intersection pairing.  For
$0\ne\alpha\in H_{n-1}(\T^n,\R)$ set
\begin{equation}\label{eq:Kalpha}
  K(\alpha)=\{k\in\Gamma:I(k,\alpha)=0\}
\end{equation}
and
\begin{equation}\label{eq:Valpha}
  V(\alpha)
  =\{\beta:I(k,\beta)=0\text{ for every }k\in K(\alpha)\}.
\end{equation}
The space $V(\alpha)$ is the smallest linear subspace containing $\alpha$
and spanned by integral classes.

We briefly recall the minimizing-lamination notation used in
\cite{JungingerGestrich2007}.  On the Abelian cover
$\R^n\to\T^n$, let $\mathcal M$ denote the nonzero homologically minimizing
codimension-one boundaries with the Birkhoff property.  Every
$R\in\mathcal M$ has a normalized first invariant
$\alpha_1(R)$ satisfying $\|\alpha_1(R)\|_g=1$, followed possibly by secondary
invariants.  For $\|\alpha\|_g=1$, the set $\mathcal M(\alpha)$ consists of the
currents whose first invariant is $\alpha$ and which have no secondary
invariant; equivalently, they are invariant under every deck transformation
in $K(\alpha)$.  The set $\mathcal M(\alpha)$ is totally ordered and gives a
minimal lamination of $\R^n$.

We use the following four facts from the cited works.  First, Auer and Bangert \cite{AuerBangert2006} prove that
$\|.\|_g|_{V(\alpha)}$ is differentiable at $\alpha$.  Second, if $\alpha$ is
rationally dependent, Junginger-Gestrich \cite{JungingerGestrich2007} proves that $\|.\|_g$ is differentiable
at $\alpha$ in the remaining directions if and only if
$\mathcal M(\alpha)$ has no gaps.  Third, on a torus every element of
$\mathcal M$ has connected support (see \cite{Nguyen2026}.  Finally, the subfamily of minimizing currents whose
supports meet a fixed compact set is compact, and the first invariant is
continuous under this convergence.

\begin{theorem}
\label{thm:foliation-differentiability}
Let $g$ be a smooth Riemannian metric on $\T^n$.  The following statements
are equivalent.
\begin{enumerate}[label=\textup{(\roman*)},leftmargin=2.3em]
\item Every primitive class in $H_{n-1}(\T^n,\Z)$ admits a foliation by
      connected homologically area-minimizing hypersurfaces.
\item The stable norm $\|.\|_{\st,g}$ is differentiable at every nonzero point of
      $H_{n-1}(\T^n,\R)$.
\end{enumerate}
In arbitrary dimension, statement \textup{(i)} may be read in the
codimension-one minimizing current sense.  For $n\le7$ the leaves are smooth
embedded hypersurfaces.
\end{theorem}

\begin{proof}
Assume first that \textup{(i)} holds.  Fix $\alpha\ne 0$ and normalize it by
$\|\alpha\|_g=1$.

If $\alpha$ is totally irrational, then $V(\alpha)=H_{n-1}(\T^n,\R)$, and the
full differentiability follows directly from Auer-Bangert \cite{AuerBangert2006}.  If the
direction of $\alpha$ is rational, let $r$ be the primitive integral class
on its ray.  The assumed foliation in the class $r$ lifts to a foliation of
$\R^n$ by elements of $\mathcal M(\alpha)$.  Hence
$\mathcal M(\alpha)$ has no gaps, and the Auer-Bangert-Junginger-Gestrich
criterion gives differentiability of $\|.\|_g$ at $\alpha$.

It remains to consider the rationally dependent case where
\[
  1<\dim V(\alpha)<n.
\]
The group $V(\alpha)\cap H_{n-1}(\T^n,\Z)$ is a full, saturated lattice in
$V(\alpha)$, and primitive lattice directions are dense in the projective
space of $V(\alpha)$.  We may therefore choose primitive integral classes
$r_j\in V(\alpha)$ such that
\begin{equation}\label{eq:alphaj-convergence}
  \alpha_j:=\frac{r_j}{\|r_j\|_g}\longrightarrow\alpha.
\end{equation}
By \eqref{eq:Valpha},
\begin{equation}\label{eq:K-inclusion}
  K(\alpha)\subset K(r_j)=K(\alpha_j)
  \qquad\text{for every }j.
\end{equation}

Fix an arbitrary point $x\in\R^n$.  A foliation of $\T^n$ by connected
closed hypersurfaces in a primitive class is a fibration over $S^1$ whose
cohomology class is Poincar\'e dual to that homology class.  On the universal
cover, its leaf parameter changes under a deck transformation $k$ by
$I(k,r_j)$.  Hence every $k\in K(r_j)$ preserves each lifted leaf.  Let
$R_j^x$ be the lifted leaf through $x$, with the orientation for which its
first invariant is $\alpha_j$.  It is a homologically minimizing Birkhoff
current and is invariant under every deck transformation in $K(r_j)$.  In
particular,
\begin{equation}\label{eq:common-periodicity}
  (\tau_k)_\#R_j^x=R_j^x
  \qquad\text{for every }k\in K(\alpha).
\end{equation}

The compactness theorem for minimizing Birkhoff currents meeting the compact
set $\{x\}$ gives, after passage to a subsequence,
\begin{equation}\label{eq:Rjx-limit}
  R_j^x\longrightarrow R^x\in\mathcal M,
  \qquad x\in\operatorname{spt}R^x.
\end{equation}
Continuity of the first invariant and \eqref{eq:alphaj-convergence} imply
\[
  \alpha_1(R^x)=\alpha.
\]
For every $k\in K(\alpha)$, equation
\eqref{eq:common-periodicity} passes to the limit, so
$(\tau_k)_\#R^x=R^x$.  The classification of Birkhoff minimizers then shows
that $R^x$ has no secondary invariant, hence
\[
  R^x\in\mathcal M(\alpha).
\]
Since $x$ was arbitrary, the supports of the currents in
$\mathcal M(\alpha)$ cover $\R^n$.  The set $\mathcal M(\alpha)$ is already
a lamination, so it is a foliation.  Junginger-Gestrich's gap criterion gives
differentiability in all directions outside $V(\alpha)$, while
Auer-Bangert gives differentiability in the directions of $V(\alpha)$.
Thus $\|.\|_g$ is differentiable at $\alpha$.  This proves \textup{(ii)}.

Conversely, assume \textup{(ii)} and let $r$ be a primitive integral class.
Set $\alpha=r/\|r\|_g$.  This is a rational direction, and full
differentiability at $\alpha$ implies, by the gap criterion, that
$\mathcal M(\alpha)$ foliates $\R^n$.  Every leaf is invariant under
$K(r)$, and the quotient by the deck group projects the lamination to a
foliation of $\T^n$ by closed homologically area-minimizing hypersurfaces in
the class $r$.  This proves \textup{(i)}.
\end{proof}

\section{Rigidity of the affine foliation condition}
\label{app:affine}

We now fix the standard affine structure on
\(
  \T^n=\R^n/\Z^n.
\)
For a primitive vector \(p\in\Z^n\), let
\begin{equation}\label{eq:affine-foliation}
  \mathcal F_p
  :=\bigl\{\Sigma_{p,s}:p\cdot x=s\pmod 1\bigr\}_{s\in S^1}
\end{equation}
be the standard affine foliation in the class Poincar\'e dual to \(p\).

\begin{theorem}
\label{thm:affine-rigidity}
Let \(n\ge2\), and let \(g\) be a smooth Riemannian metric on the standard
affine torus \(\T^n\).  Suppose that for every primitive covector
\(p\in\Z^n\), every leaf of the affine foliation \(\mathcal F_p\) is
\(g\)-minimal.  Then the coefficients of \(g\) are constant in the
standard affine coordinates.  In particular, \(g\) is flat.
\end{theorem}

 The proof combines Hangan's local
classification of metrics with minimal affine hyperplanes in \cite{Hangan1996} with a global
periodicity argument.  Hangan first derived the differential system imposed
by the vanishing of the mean curvature of all affine hyperplanes and then
integrated it in terms of Euclidean Killing $2$-tensors.

\begin{proof}
Lift \(g\) to a \(\Z^n\)-periodic metric (still denoted by $g$) on \(\R^n\).  For a nonzero constant vector \(p\in \R ^n\), let
\[
  u_p(x)=p\cdot x.
\]
By the assumption, using compactness property of stable minimal hypersurfaces, every affine hyperplane is minimal. The mean curvature of the level hyperplane through \(x\), computed with the
unit normal \(\nabla^g u_p/|\nabla^g u_p|_g\), is
\[
  H_g(x,p)
  =\operatorname{div}_g
   \left(\frac{\nabla^g u_p}{|\nabla^g u_p|_g}\right)=0.
\]

Write
\[
  g=g_{ij}(x)\,dx^i dx^j,
  \qquad
  |g|:=\det(g_{ij}).
\]
Hangan's correspondence now applies.  In the notation of
\cite{Hangan1996}, the positive-definite symmetric tensor
\begin{equation}\label{eq:Hangan-normalization}
  G:=|g|^{-\frac{2}{n+1}}g
\end{equation}
satisfies the Euclidean Killing-tensor equation
\begin{equation}\label{eq:flat-Killing-equation}
  \partial_iG_{jk}+\partial_jG_{ki}+\partial_kG_{ij}=0.
\end{equation}

The Killing equation has a direct dynamical meaning.  For every constant
vector \(v\in\R^n\) and every Euclidean straight line
\(\gamma(t)=x+tv\), equation \eqref{eq:flat-Killing-equation} gives
\begin{equation}\label{eq:quadratic-integral}
  \frac{d}{dt}G_{x+tv}(v,v)
  =(\partial_vG)_{x+tv}(v,v)=0.
\end{equation}
Thus \(G(v,v)\) is a quadratic first integral of the standard straight-line
flow.

Choose \(v\) whose coordinates are rationally independent.  The projected
line \(x+tv\pmod{\Z^n}\) is dense in \(\T^n\).  By
\eqref{eq:quadratic-integral}, the continuous function
\(G_x(v,v) \)
is constant along this dense orbit and hence constant on all of \(\T^n\). Since totally irrational vectors are dense in \(\R^n\), we have that for fixed
\(x,y\in\T^n\), the quadratic polynomial
\[
  v\longmapsto G_x(v,v)-G_y(v,v)
\]
vanishes on a dense subset of \(\R^n\) and hence vanishes identically.
Polarization gives
\[
  G_x(v,w)=G_y(v,w)
  \qquad
  \text{for all }x,y\in\T^n
  \text{ and }v,w\in\R^n.
\]
Consequently, \(G\) has constant coefficients.  This implies that \(g=|G|^{\frac{2}{1-n}}G\) also has constant coefficients and is flat. This completes the proof.
\end{proof}

\end{document}